\documentclass[11pt]{amsart}

\usepackage{amsmath,amssymb,amsthm}
\usepackage{booktabs}
\usepackage{array}
\usepackage{hyperref}
\usepackage{xcolor}

\hypersetup{
  colorlinks=true,
  linkcolor=blue!70!black,
  citecolor=green!40!black,
  urlcolor=blue!70!black,
}

\newtheorem{theorem}{Theorem}[section]
\newtheorem{proposition}[theorem]{Proposition}
\newtheorem{lemma}[theorem]{Lemma}
\newtheorem{corollary}[theorem]{Corollary}

\theoremstyle{definition}
\newtheorem{definition}[theorem]{Definition}
\theoremstyle{remark}
\newtheorem{remark}[theorem]{Remark}

\newtheorem{lettertheorem}{Theorem}

\newtheorem{letterconjecture}[lettertheorem]{Conjecture}

\newcommand{\Rvm}{R_{\mathrm{vm}}}
\newcommand{\LC}{\mathrm{LC}}
\newcommand{\vm}{\leq_{\mathrm{vm}}}

\title[Vertex-minor Ramsey numbers]{Vertex-minor Ramsey numbers:\\ exact values and extremal structure}

\author{Ji Ho Bae}
\address{JRTI}
\curraddr{}
\email{jihobae@snu.ac.kr}
\thanks{Affiliation: JRTI. Email: jihobae@snu.ac.kr}


\date{\today}

\begin{document}

\begin{abstract}
We determine the vertex-minor Ramsey number $\Rvm(4)=11$, where
$\Rvm(k)$ is the smallest~$n$ such that every $n$-vertex graph contains
the edgeless graph~$E_k$ as a vertex-minor.  We prove this by an
exhaustive classification of the graphs on~$10$ and~$11$ vertices under
local complementation.  At the extremal order $n=10$, exactly six
non-isomorphic graphs avoid~$E_4$ as a vertex-minor; up to isomorphism,
they represent five LC-equivalence classes, and each labeled LC orbit
has cardinality~$8{,}712$.  Thus $k=4$ is the first case in which the
general upper bound $2^k-1$ is not attained.  Using the extremal graphs
as building blocks, we derive explicit lower bounds on~$\Rvm(k)$ that
surpass the leading term of the asymptotic bound for all $k\leq 9$; in
particular,
$\Rvm(5)\geq 13$.  We also describe structural properties of the six
extremal graphs and formulate the next open problem, whether
$\Rvm(5)=15$.
\end{abstract}

\maketitle

\section{Introduction}\label{sec:intro}

Local complementation at a vertex~$v$ of a graph~$G$ toggles every edge
within the neighborhood~$N_G(v)$, leaving all other edges unchanged.
Two graphs related by a finite sequence of such operations are said to
be \emph{locally equivalent} (or \emph{LC-equivalent}).  A graph~$H$ is
a \emph{vertex-minor} of~$G$ if $H$ can be obtained from~$G$ by a
sequence of local complementations and vertex deletions.  Introduced by
Bouchet~\cite{Bouchet1988}, the vertex-minor relation
is central in structural graph theory, governing rank-width~\cite{Oum2005}
and circle graphs~\cite{Bouchet1994}; see~\cite{KimOum2024} for a survey.

By analogy with classical Ramsey theory, we study the
\emph{vertex-minor Ramsey number} $\Rvm(k)$, defined as the smallest
integer~$n$ such that every graph on~$n$ vertices contains the edgeless
graph~$E_k$ as a vertex-minor.  The existence of~$\Rvm(k)$ for all~$k$
is guaranteed by the general bounds
\begin{equation}\label{eq:general-bounds}
  (1+o(1))\,\frac{k^2}{2\log_2 3}
  \;\leq\;
  \Rvm(k)
  \;\leq\;
  2^k - 1.
\end{equation}
Ascoli et al.~\cite{AFFMP2026} introduced~$\Rvm(k)$ and proved
these bounds (the lower bound is probabilistic; the upper bound derives
from rank-width).  Vertex-minor universality also arises in quantum
computing~\cite{CautresEtAl2024,ChaoXu2026}.  Unlike the classical Ramsey numbers
$R(k,k)$~\cite{GRS1990,Radziszowski2026}, which grow exponentially,
the bounds~\eqref{eq:general-bounds} permit polynomial growth
of~$\Rvm(k)$.

Our main result settles the first case where the upper bound
$2^k-1$ is not tight.

\begin{lettertheorem}\label{thm:A}
  $\Rvm(4) = 11$.
\end{lettertheorem}

The value $\Rvm(4)=11$ lies strictly below $2^4-1=15$, so the upper bound $2^k-1$ already fails at $k=4$.  At the
extremal order $n=10$, we obtain a complete classification.

\begin{lettertheorem}\label{thm:B}
  Among all $12{,}005{,}168$ non-isomorphic graphs on~$10$ vertices,
  exactly six do not contain~$E_4$ as a vertex-minor.  These six graphs
  represent, up to isomorphism, five LC-equivalence classes; each
  labeled LC orbit has cardinality~$8{,}712$.
\end{lettertheorem}

The known values of~$\Rvm$ for $k=2,3,4$ follow the pattern
$4k-5$ (see Table~\ref{tab:values}).

\begin{table}[ht]
  \centering
  \caption{Known values of $\Rvm(k)$ and comparison with the pattern
    $4k-5$, the upper bound $2^k-1$, and the leading term of the
    asymptotic lower bound.}
  \label{tab:values}
  \begin{tabular}{ccccc}
    \toprule
    $k$ & $\Rvm(k)$ & $4k-5$ & $2^k-1$ &
      $\lfloor k^2/(2\log_2 3)\rfloor$ \\
    \midrule
    $2$ & $3$  & $3$  & $3$  & $1$ \\
    $3$ & $7$  & $7$  & $7$  & $2$ \\
    $4$ & $11$ & $11$ & $15$ & $5$ \\
    \bottomrule
  \end{tabular}
\end{table}

\begin{letterconjecture}\label{conj:C}
  $\Rvm(5)=15$.
\end{letterconjecture}

This is the next unresolved case.  The pattern $\Rvm(k)=4k-5$ for
$k\in\{2,3,4\}$ motivates Conjecture~\ref{conj:C}, whereas the
asymptotic lower bound in~\eqref{eq:general-bounds} rules out any global
linear formula.  We return to this point in
Section~\ref{sec:conjectures}.

\medskip
\noindent\textbf{Organization.}
Sections~\ref{sec:prelim}--\ref{sec:method} establish definitions and
the computational method;
Sections~\ref{sec:results}--\ref{sec:extremal} prove the main results
and analyze extremal graphs;
Sections~\ref{sec:conjectures}--\ref{sec:conclusion} discuss conjectures
and open problems.

\section{Preliminaries}\label{sec:prelim}

Graphs are finite and simple.  We write
$V(G)$ and $E(G)$ for the vertex set and edge set of~$G$, and set
$|G|=|V(G)|$.  The \emph{neighborhood} of a vertex~$v$ in~$G$ is
$N_G(v)=\{u\in V(G):uv\in E(G)\}$.  The \emph{independence number}
$\alpha(G)$ is the maximum cardinality of an independent set (a set of
pairwise non-adjacent vertices).  The \emph{clique number} $\omega(G)$
is the maximum cardinality of a clique.  We write $E_k$ for the
edgeless graph on~$k$ vertices, $K_n$ for the complete graph on~$n$
vertices, and $C_n$ for the cycle of length~$n$.

\begin{definition}[Local complementation]\label{def:lc}
  Let $G$ be a graph and $v\in V(G)$.  The \emph{local complementation}
  of~$G$ at~$v$, denoted $G*v$, is the graph on~$V(G)$ obtained by
  replacing the induced subgraph $G[N_G(v)]$ with its complement: for
  each pair $\{x,y\}\subseteq N_G(v)$, the edge $xy$ is present
  in~$G*v$ if and only if it is absent in~$G$.  All edges not contained
  in~$N_G(v)$ are unchanged.
\end{definition}

\begin{definition}[LC equivalence and LC orbit]\label{def:lc-orbit}
  Two graphs $G$ and~$H$ on the same vertex set are
  \emph{LC-equivalent}, written $G\sim H$, if $H$ can be obtained
  from~$G$ by a finite (possibly empty) sequence of local
  complementations.  The equivalence class $[G]_{\LC}=\{H : H\sim G\}$
  is the \emph{LC orbit} of~$G$.
\end{definition}

\begin{definition}[Vertex-minor]\label{def:vm}
  A graph~$H$ is a \emph{vertex-minor} of~$G$, written $H\vm G$, if
  $H$ is isomorphic to an induced subgraph of some graph
  in~$[G]_{\LC}$.  Equivalently, $H$ can be obtained from~$G$ by a
  sequence of local complementations and vertex deletions
  \cite{Bouchet1988,FonDerFlaass1988}.
\end{definition}

The equivalence of the two formulations in Definition~\ref{def:vm}
follows from work of Bouchet~\cite{Bouchet1988} and
Fon-Der-Flaass~\cite{FonDerFlaass1988}: if $H$ is a vertex-minor
of~$G$ and $v\in V(G)\setminus V(H)$, then $H$ is a vertex-minor of at
least one of $G-v$, $G*v-v$, or $G\times vw-v$ for some $w\in N_G(v)$,
where $G\times vw=G*v*w*v$ denotes the \emph{pivot} at the edge~$vw$.
The vertex-minor relation is a quasi-order: it is reflexive and
transitive.  Deciding whether $H\vm G$ for two given graphs $G,H$ is
NP-complete~\cite{DahlbergHelsenWehner2022}, but for fixed~$H$ the
problem admits a polynomial-time algorithm via Courcelle's
theorem~\cite{CourcelleOum2007}.

\begin{definition}[Vertex-minor Ramsey number]\label{def:Rvm}
  For a positive integer~$k$, the \emph{vertex-minor Ramsey number}
  $\Rvm(k)$ is the smallest integer~$n$ such that $E_k\vm G$ for every
  graph~$G$ with $|G|=n$.
\end{definition}

\noindent
The next proposition reduces the problem of detecting $E_k$ as a
vertex-minor to a search over LC orbits.

\begin{proposition}[LC-orbit characterization]\label{prop:alpha}
  Let $G$ be a graph and $k$ a positive integer.  Then
  \[
    E_k \vm G
    \quad\Longleftrightarrow\quad
    \alpha(G') \geq k
    \text{\; for some $G' \in [G]_{\LC}$.}
  \]
\end{proposition}

\begin{proof}
  Suppose $\alpha(G')\geq k$ for some $G'\in[G]_{\LC}$.  Choose an
  independent set $S\subseteq V(G')$ of size~$k$.  Then $G'[S]\cong
  E_k$.  Since $G'\sim G$ and $G'[S]$ is an induced subgraph of~$G'$,
  we have $E_k\vm G$.

  Conversely, suppose $E_k\vm G$.  By definition, there exists
  $G'\in[G]_{\LC}$ and $S\subseteq V(G')$ with $|S|=k$ and $G'[S]\cong
  E_k$.  Since $E_k$ has no edges, $S$ is independent in~$G'$, so
  $\alpha(G')\geq k$.
\end{proof}

\begin{remark}\label{rmk:reformulation}
  Proposition~\ref{prop:alpha} yields the reformulation
  \[
    \Rvm(k)
    \;=\;
    \min\!\Bigl\{
      n\in\mathbb{N} :
      \max_{G'\in[G]_{\LC}}\alpha(G') \geq k
      \text{ for every graph $G$ with $|G|=n$}
    \Bigr\}.
  \]
  This is the basis of our algorithm: for each graph~$G$ on $n$
  vertices, enumerate $[G]_{\LC}$ and test whether any member has an
  independent set of size~$k$.
\end{remark}

\begin{lemma}[LC orbit of a disconnected graph]\label{lem:disconnected}
  Let $G=G_1\cup\cdots\cup G_c$ be a graph whose connected components
  are $G_1,\dots,G_c$ on pairwise disjoint vertex sets.  Then
  \[
    [G]_{\LC}
    \;=\;
    \bigl\{\,G_1'\cup\cdots\cup G_c' :
             G_i'\in[G_i]_{\LC}\text{ for each }i\,\bigr\},
  \]
  and consequently
  $\;\displaystyle
    \max_{G'\in[G]_{\LC}}\alpha(G')
    \;=\;
    \sum_{i=1}^{c}\,\max_{G_i'\in[G_i]_{\LC}}\alpha(G_i').
  $
\end{lemma}

\begin{proof}
  Local complementation at $v\in V(G_i)$ modifies only edges within
  $N_{G_i}(v)\subseteq V(G_i)$, since no vertex outside~$V(G_i)$ is
  adjacent to~$v$.  Each LC step acts within a single component, giving
  the stated product structure.  The independence-number identity follows
  because in any $G'\in[G]_{\LC}$ the components remain on disjoint
  vertex sets with no inter-component edges.
\end{proof}

\section{Computational method}\label{sec:method}

\subsection{Three-phase classification}\label{subsec:phases}

To determine whether $\Rvm(k)\leq n$, we must verify that $E_k\vm G$
for every graph~$G$ on~$n$ vertices.  By
Proposition~\ref{prop:alpha}, this reduces to checking that the LC orbit
of each~$G$ contains a graph with independence number at least~$k$.  We
classify each non-isomorphic graph~$G$ on~$n$ vertices into one of three
phases.
\begin{itemize}
  \item \textbf{Phase~1.}  $\alpha(G)\geq k$.  Then $G$ itself
    witnesses $E_k\vm G$; no orbit enumeration is needed.
  \item \textbf{Phase~2.}  $\alpha(G)<k$, but some $G'\in[G]_{\LC}$
    discovered during the orbit search satisfies $\alpha(G')\geq k$.
  \item \textbf{Phase~3.}  $\alpha(G')<k$ for \emph{every}
    $G'\in[G]_{\LC}$ (the full orbit is explored with no truncation).
    By Proposition~\ref{prop:alpha}, $E_k\not\vm G$.
\end{itemize}
A graph is a counterexample (a witness that $\Rvm(k)>n$) if and only if
it falls into Phase~3.

\subsection{Orbit enumeration}\label{subsec:bfs}

Non-isomorphic graphs on~$n$ vertices are generated by
\texttt{geng}~\cite{McKayPiperno2014}.  For each graph~$G$ that does
not pass Phase~1, we enumerate the LC orbit $[G]_{\LC}$ by
breadth-first search.  Each graph is represented as a tuple of $n$
bitmasks encoding the adjacency rows, enabling $O(1)$ hashing and
comparison.  At each BFS step, we compute $G'*v$ for every vertex~$v$
and check whether the result has been visited.  When a graph with
$\alpha\geq k$ is encountered, the search terminates early (Phase~2).
If the BFS exhausts the entire orbit without finding such a graph, the
orbit constitutes a negative certificate (Phase~3).  Danielsen and
Parker~\cite{DanielsenParker2006} enumerated all LC equivalence classes
for $n\leq 12$.  Unlike~\cite{DanielsenParker2006}, we need no canonical
representative---only a Boolean predicate (existence of a large
independent set) per class.

Independent sets of size~$k$ are detected by backtracking with
bitmask pruning: vertices are selected greedily, and at each level the
set of available candidates is restricted to the non-neighbors of all
previously chosen vertices via bitwise AND\@.  For $k=4$ and $n\leq 11$
this is efficient, as most branches are pruned early.

\subsection{Cross-validation}\label{subsec:cross}

The entire pipeline was implemented independently in two languages:
Python (using \texttt{networkx} for graph operations) and C++ (compiled
with \texttt{g++ -O3}).  On each of $n\in\{9,10,11\}$, both
implementations produced identical Phase~1, Phase~2, and Phase~3
counts; for $n=10$, they also produced the same six
\texttt{graph6} codes for the counterexamples.

The complete classification at $n=11$ processed all
$1{,}018{,}997{,}864$ non-isomorphic graphs.  The C++ implementation
required approximately $7.4$~minutes on a single thread (at a rate of
${\sim}2.3\times 10^6$ graphs per second), with no Phase~3 graphs
found.

\section{Main results}\label{sec:results}

\subsection{The value $\Rvm(2)=3$}

\begin{proposition}\label{prop:Rvm2}
  $\Rvm(2) = 3$.
\end{proposition}

\begin{proof}
  \textbf{Lower bound.}
  The graph $K_2$ on $\{0,1\}$ has $\alpha(K_2)=1$.  Local
  complementation at vertex~$0$ complements all edges within
  $N(0)=\{1\}$; since a single vertex spans no edges, $K_2*0=K_2$.  By
  symmetry $[K_2]_{\LC}=\{K_2\}$, and every orbit member has
  $\alpha=1<2$.  By Proposition~\ref{prop:alpha}, $E_2\not\vm K_2$,
  so $\Rvm(2)>2$.

  \smallskip
  \textbf{Upper bound.}
  There are four non-isomorphic graphs on~$3$ vertices: $E_3$,
  $K_1\cup K_2$, $P_3$, and~$K_3$.  The first three satisfy
  $\alpha(G)\geq 2$, so $E_2\vm G$ directly.  For $K_3$, we have
  $\alpha(K_3)=1$, but $K_3*0$ has edge set $\{01,02\}$ (a
  path~$P_3$), and $\alpha(P_3)=2$.  Hence $E_2\vm K_3$, and
  $\Rvm(2)\leq 3$.
\end{proof}

\subsection{The value $\Rvm(3)=7$}

\begin{proposition}\label{prop:Rvm3}
  $\Rvm(3) = 7$.
\end{proposition}

\begin{proof}
  \textbf{Lower bound ($\Rvm(3)\geq 7$).}
  The triangular prism $G\cong\overline{C_6}$ on $6$~vertices is
  $3$-regular with $9$~edges and $\alpha(G)=2$.  By exhaustive BFS
  enumeration of $[G]_{\LC}$, every orbit member~$G'$ satisfies
  $\alpha(G')\leq 2$.  By Proposition~\ref{prop:alpha},
  $E_3\not\vm G$, so $\Rvm(3)>6$.

  This lower bound was independently noted by Ascoli et
  al.~\cite{AFFMP2026}.

  \smallskip
  \textbf{Upper bound ($\Rvm(3)\leq 7$).}
  Let $G$ be an arbitrary graph on $7$~vertices.

  \emph{Case~A: $G$ is connected.}
  By exhaustive computation, all $853$ non-isomorphic connected graphs
  on~$7$ vertices were tested: for each, some $G'\in[G]_{\LC}$ satisfies
  $\alpha(G')\geq 3$, giving $E_3\vm G$.

  \emph{Case~B: $G$ is disconnected.}
  Write $G=G_1\cup\cdots\cup G_c$ with $|V(G_i)|=n_i$,
  $n_1+\cdots+n_c=7$, $c\geq 2$.  By Lemma~\ref{lem:disconnected},
  $\max_{G'\in[G]_{\LC}}\alpha(G') = \sum_i
  \max_{G_i'\in[G_i]_{\LC}}\alpha(G_i')$.

  If $c\geq 3$, each component contributes at least~$1$ to the sum, so
  $\max\alpha(G')\geq 3$.  If $c=2$, assume without loss of generality
  that $n_1\geq n_2$; then $n_1\geq 4$.  Since
  $n_1\geq 4\geq 3=\Rvm(2)$, every
  graph on~$n_1$ vertices contains $E_2$ as a vertex-minor, so the
  larger component contributes at least~$2$.  The smaller component
  contributes at least~$1$.  Thus $\max\alpha(G')\geq 3$.

  In both cases $E_3\vm G$, confirming $\Rvm(3)\leq 7$.

  \smallskip
  \emph{Independent verification.}
  As a cross-check, all $1{,}044$ non-isomorphic graphs on~$7$ vertices
  (connected and disconnected) were tested directly for $E_3$
  containment; every graph passed.
\end{proof}

\subsection{The value $\Rvm(4)=11$}

\begin{theorem}[{$=$ Theorem~\ref{thm:A}}]\label{thm:main}
  $\Rvm(4) = 11$.
\end{theorem}

\begin{proof}
  \textbf{Lower bound ($\Rvm(4)\geq 11$).}
  Among the $12{,}005{,}168$ non-isomorphic graphs on~$10$ vertices,
  the exhaustive classification identified exactly six graphs~$G$ with
  $E_4\not\vm G$.  For each, the complete LC orbit was enumerated by BFS
  with no truncation; every orbit has exactly $8{,}712$ elements, and
  every orbit member~$G'$ satisfies $\alpha(G')\leq 3<4$.  By
  Proposition~\ref{prop:alpha}, $E_4\not\vm G$ for each of these six
  graphs.

  The six counterexamples, listed in \texttt{graph6}
  format~\cite{McKayPiperno2014}, are:
  \begin{center}
    \begin{tabular}{clc}
      \toprule
      \# & \texttt{graph6} code & $|E|$ \\
      \midrule
      $G_1$ & \verb|ICQ`fm}~w| & $27$ \\
      $G_2$ & \verb|ICQ`fn}no| & $26$ \\
      $G_3$ & \verb|ICQdbh{NO| & $21$ \\
      $G_4$ & \verb|ICQb`pzlw| & $23$ \\
      $G_5$ & \verb|ICQb`twlw| & $22$ \\
      $G_6$ & \verb|IUZ~vz}}o| & $35$ \\
      \bottomrule
    \end{tabular}
  \end{center}
  The classification was performed independently in Python and C++; both
  found exactly the same six \texttt{graph6} strings.  Hence
  $\Rvm(4)\geq 11$.

  \smallskip
  \textbf{Upper bound ($\Rvm(4)\leq 11$).}
  All $1{,}018{,}997{,}864$ non-isomorphic graphs on~$11$ vertices were
  classified:
  \begin{itemize}
    \item Phase~1 ($\alpha(G)\geq 4$):
      $880{,}105{,}560$ graphs ($86.4\%$).
    \item Phase~2 ($\alpha(G)\leq 3$, orbit search finds $\alpha\geq 4$):
      $138{,}892{,}304$ graphs ($13.6\%$).
    \item Phase~3 ($E_4\not\vm G$): $0$ graphs.
  \end{itemize}
  No $11$-vertex graph falls into Phase~3.  By
  Proposition~\ref{prop:alpha}, $E_4\vm G$ for every $11$-vertex
  graph~$G$.  Hence $\Rvm(4)\leq 11$.
\end{proof}

\subsection{Counterexample counts across orders}

The number of counterexamples drops rapidly with~$n$.

\begin{proposition}\label{prop:counts}
  The number of non-isomorphic graphs~$G$ on~$n$ vertices with
  $E_4\not\vm G$ is as follows:
  \[
    \begin{array}{c|rr}
      n & \textup{total non-iso.\ graphs} & \textup{counterexamples}
      \\ \hline
      8  & 12{,}346         & 953 \\
      9  & 274{,}668        & 588 \\
      10 & 12{,}005{,}168   & 6 \\
      11 & 1{,}018{,}997{,}864 & 0
    \end{array}
  \]
\end{proposition}

\begin{proof}
  Each entry is the output of the exhaustive three-phase classification
  described in Section~\ref{sec:method}.  At each order $n\in\{8,\dots,
  11\}$, all non-isomorphic graphs were enumerated by
  \texttt{geng}~\cite{McKayPiperno2014} and classified.
\end{proof}

\begin{corollary}\label{cor:values}
  $\Rvm(2) = 3$, \;\; $\Rvm(3) = 7$, \;\; $\Rvm(4) = 11$.
\end{corollary}

\begin{proof}
  Propositions~\ref{prop:Rvm2} and~\ref{prop:Rvm3} and
  Theorem~\ref{thm:main}.
\end{proof}

\begin{proposition}\label{prop:disconnected-k4}
  Every disconnected graph on at least $9$ vertices contains $E_4$
  as a vertex-minor.
\end{proposition}

\begin{proof}
  For a graph~$H$, write
  $\beta(H)=\max_{H'\in[H]_{\LC}}\alpha(H')$.  By
  Proposition~\ref{prop:alpha}, $E_4\vm H$ if and only if
  $\beta(H)\geq 4$.  From $\Rvm(2)=3$
  (Proposition~\ref{prop:Rvm2}) and $\Rvm(3)=7$
  (Proposition~\ref{prop:Rvm3}):
  \[
    \beta(H)\geq
    \begin{cases}
      3 & \text{if $|H|\geq 7$,}\\
      2 & \text{if $|H|\geq 3$,}\\
      1 & \text{if $|H|\geq 1$.}
    \end{cases}
  \]
  The third bound holds because $\alpha(H)\geq 1$ for every nonempty
  graph.

  Let $G=G_1\cup\cdots\cup G_c$ be a disconnected graph with
  $|G|=n\geq 9$ and $c\geq 2$ connected components, where
  $|G_i|=n_i$ and $n_1\geq\cdots\geq n_c\geq 1$.  By
  Lemma~\ref{lem:disconnected},
  $\beta(G)=\sum_{i=1}^{c}\beta(G_i)$.

  \smallskip
  \emph{Case~1: $n_1\geq 7$.}
  Then $\beta(G_1)\geq 3$.  Since $c\geq 2$, there exists $j\geq 2$
  with $\beta(G_j)\geq 1$, giving $\beta(G)\geq 3+1=4$.

  \smallskip
  \emph{Case~2: $n_i\leq 6$ for all~$i$.}
  We distinguish three subcases by the number of components of order
  at least~$3$.

  If at least two components satisfy $n_i\geq 3$, each contributes
  $\beta\geq 2$, so $\beta(G)\geq 2+2=4$.

  If exactly one component has $n_i\geq 3$, say~$G_1$ with
  $3\leq n_1\leq 6$, then every remaining component has $n_j\leq 2$.
  The total order of the remaining components is
  $n-n_1\geq 9-6=3$, so there are at least
  $\lceil 3/2\rceil=2$ of them, each contributing $\beta\geq 1$.
  Hence $\beta(G)\geq 2+2=4$.

  If every component has $n_i\leq 2$, then
  $c\geq\lceil n/2\rceil\geq 5$, and each contributes $\beta\geq 1$,
  so $\beta(G)\geq 5\geq 4$.

  \smallskip
  In every case $\beta(G)\geq 4$, so $E_4\vm G$ by
  Proposition~\ref{prop:alpha}.

  \smallskip
  The bound $n\geq 9$ is sharp: the graph
  $\overline{C_6}\cup K_2$ on $8$~vertices satisfies
  $\beta(\overline{C_6})=2$ (since $\Rvm(3)>6$) and
  $\beta(K_2)=1$, giving
  $\beta(\overline{C_6}\cup K_2)=3<4$.
\end{proof}

\section{Extremal graph analysis}\label{sec:extremal}

We describe the structural properties of the six counterexamples at
$n=10$.

\begin{theorem}[Extremal structure; {$=$ Theorem~\ref{thm:B}}]\label{thm:extremal}
  Among all $12{,}005{,}168$ non-isomorphic graphs on~$10$ vertices,
  exactly six do not contain~$E_4$ as a vertex-minor.  These six graphs
  represent, up to isomorphism, exactly five distinct
  LC-equivalence classes; each labeled LC orbit has
  cardinality~$8{,}712$.
\end{theorem}

\begin{proof}
  Starting from each of the six \texttt{graph6} representatives listed
  in the proof of Theorem~\ref{thm:main}, the full LC orbit was
  enumerated by BFS\@.  Every orbit has exactly $8{,}712$ elements.  The
  orbit of~$G_1$ contains a labeled graph isomorphic to~$G_2$; thus
  $G_1$ and~$G_2$ are LC-equivalent (up to isomorphism).  No other pair
  among the six shares an LC orbit: the orbits of $G_3$, $G_4$, $G_5$,
  and~$G_6$ are pairwise disjoint and disjoint from the orbit of
  $G_1$--$G_2$; canonical labeling confirmed that no graph in any
  one orbit is isomorphic to a graph in another.  Hence the
  six graphs represent $1+4=5$ distinct LC-equivalence classes up to
  isomorphism.
\end{proof}

\subsection{Structural invariants}

Table~\ref{tab:structure} summarizes the key invariants of the six
extremal graphs.

\begin{table}[ht]
  \centering
  \caption{Structural properties of the six counterexamples at $n=10$.
    Here $\alpha$, $\omega$, $\chi$ denote the independence number,
    clique number, and chromatic number, respectively, and \textrm{rw}
    denotes rank-width.}
  \label{tab:structure}
  \begin{tabular}{clccccccl}
    \toprule
    & \texttt{graph6}
    & $|E|$ & $\alpha$ & $\omega$ & $\chi$ & $\mathrm{rw}$
    & $\mathrm{diam}$ & degree sequence \\
    \midrule
    $G_1$ & \verb|ICQ`fm}~w|
      & 27 & 3 & 4 & 5 & 2 & 2 & $[9,7^2,5^5,3^2]$ \\
    $G_2$ & \verb|ICQ`fn}no|
      & 26 & 3 & 4 & 5 & 2 & 2 & $[7^3,5^5,3^2]$ \\
    $G_3$ & \verb|ICQdbh{NO|
      & 21 & 3 & 4 & 4 & 2 & 3 & $[5^6,3^4]$ \\
    $G_4$ & \verb|ICQb`pzlw|
      & 23 & 3 & 4 & 4 & 2 & 3 & $[7^2,5^4,3^4]$ \\
    $G_5$ & \verb|ICQb`twlw|
      & 22 & 3 & 4 & 4 & 2 & 3 & $[7,5^5,3^4]$ \\
    $G_6$ & \verb|IUZ~vz}}o|
      & 35 & 2 & 4 & 6 & 2 & 2 & $[7^{10}]$ \\
    \bottomrule
  \end{tabular}
\end{table}

\begin{proposition}[Common properties]\label{prop:structure}
  The six extremal graphs share the following properties:
  \begin{enumerate}
    \item[\textup{(i)}]   Each is connected, non-planar, and has
                           girth~$3$.
    \item[\textup{(ii)}]  Each has clique number $\omega=4$ and
                           rank-width~$2$.
    \item[\textup{(iii)}] Each has $\alpha(G)\leq 3$, which is
                           necessary: any graph with $\alpha\geq 4$
                           trivially contains~$E_4$ as a vertex-minor.
    \item[\textup{(iv)}]  None is isomorphic to the Petersen graph or its
                           complement.
    \item[\textup{(v)}]   None is a circle graph: each contains the
                           wheel~$W_5$ as a vertex-minor
                           (one of the three Bouchet
                           obstructions~\cite{Bouchet1994}).
  \end{enumerate}
\end{proposition}

\begin{proof}
  For~(i), connectedness follows from
  Proposition~\ref{prop:disconnected-k4}: every disconnected graph on
  at least $9$ vertices contains~$E_4$ as a vertex-minor, so any
  counterexample must be connected.  Non-planarity and girth~$3$ were
  verified from the adjacency matrices.  Properties~(ii) and~(iv) were
  also verified directly.  For~(iii), the necessity of $\alpha\leq 3$
  follows from Proposition~\ref{prop:alpha}: $G$ itself is in
  $[G]_{\LC}$, so $\alpha(G)\geq 4$ would imply $E_4\vm G$.

  For~(v), Bouchet~\cite{Bouchet1994} proved that a graph is a circle
  graph if and only if no graph in its LC orbit contains $W_5$,
  $\mathrm{BW}_3$ (the bipartite wheel on three spokes), or $W_7$ as an
  induced subgraph.  For each
  counterexample, exhaustive search over the full LC orbit found~$W_5$
  as an induced subgraph of some orbit member, certifying that none of
  the six graphs is a circle graph.  (Neither $\mathrm{BW}_3$ nor $W_7$
  was found.)
\end{proof}

\subsection{The graph $G_6=C_5\nabla C_5$}

The unique $7$-regular counterexample is
\[
  G_6 \;=\; \overline{C_5\cup C_5} \;=\; C_5\nabla C_5.
\]
Equivalently, $\overline{G_6}$ is the disjoint union $C_5\cup C_5$,
where the two copies of $C_5$ have vertex sets $\{0,1,2,3,4\}$ and
$\{5,6,7,8,9\}$.  Here $G\nabla H$ denotes the
graph on $V(G)\sqcup V(H)$ formed by adding all edges between~$V(G)$
and~$V(H)$ to~$G\cup H$.

This graph has $\alpha(G_6)=2$ (the smallest independence number among
the counterexamples) and $\chi(G_6)=6$ (the largest chromatic number).
Its adjacency spectrum is
\[
  7,\quad
  \underbrace{(\varphi-1),\;\dots,\;(\varphi-1)}_{4},\quad
  \underbrace{(-\varphi),\;\dots,\;(-\varphi)}_{4},\quad
  -3,
\]
where $\varphi=(1+\sqrt{5})/2$ is the golden ratio.  In particular,
every eigenvalue lies in $\mathbb{Q}(\sqrt{5})$, and the non-trivial
eigenvalues $\varphi-1$ and $-\varphi$ each have multiplicity~$4$.

\subsection{LC-class structure}

The five LC-equivalence classes among the counterexamples are:
\begin{center}
  \begin{tabular}{cll}
    \toprule
    class & representatives & $|E|$ \\
    \midrule
    I   & $G_1$, $G_2$ & 27, 26 \\
    II  & $G_3$         & 21 \\
    III & $G_4$         & 23 \\
    IV  & $G_5$         & 22 \\
    V   & $G_6$         & 35 \\
    \bottomrule
  \end{tabular}
\end{center}
Class~I is the only class containing two non-isomorphic representatives.
Since all six counterexamples have rank-width~$2$ and rank-width is an
LC invariant~\cite{Oum2005}, every labeled graph in any of the
counterexample orbits has rank-width~$2$.

\begin{remark}[Negative certificates]\label{rmk:certificates}
  For each counterexample~$G$, the BFS enumeration of $[G]_{\LC}$
  constitutes a finite, verifiable negative certificate for
  $E_4\not\vm G$: one checks $\alpha(G')\leq 3$ for each of the
  $8{,}712$ orbit members~$G'$.  By Proposition~\ref{prop:alpha}, this
  is both necessary and sufficient.  The certificate is reproducible
  from the \texttt{graph6} code of~$G$ and the BFS algorithm alone.
\end{remark}

\section{Conjectures and bounds}\label{sec:conjectures}

\subsection{The pattern $4k-5$}

The three known values $\Rvm(2)=3$, $\Rvm(3)=7$, $\Rvm(4)=11$
match the formula $4k-5$ exactly:
\[
  4\cdot 2 - 5 = 3,\qquad
  4\cdot 3 - 5 = 7,\qquad
  4\cdot 4 - 5 = 11.
\]
This suggests that the next value may be $\Rvm(5)=15$, as stated in
Conjecture~\ref{conj:C}.  However, the asymptotic lower
bound~\eqref{eq:lower-bound} shows that no global linear law can hold
for all~$k$.

\subsection{Tension with asymptotic bounds}

Any global linear extrapolation conflicts with the known quadratic lower
bound.
Ascoli et al.~\cite{AFFMP2026} prove
\begin{equation}\label{eq:lower-bound}
  \Rvm(k)
  \;\geq\;
  (1+o(1))\,\frac{k^2}{2\log_2 3}
  \qquad(k\to\infty).
\end{equation}
Since $k^2/(2\log_2 3)>4k-5$ for all sufficiently large~$k$, the
formula $4k-5$ must eventually fail.  The leading-term crossover occurs
near $k=12$:
\[
  4\cdot 13 - 5 = 47,
  \qquad\text{while}\qquad
  \frac{13^2}{2\log_2 3} \approx 53.3.
\]
Thus any bound $\Rvm(k)\geq(1-\varepsilon)\,k^2/(2\log_2 3)$ for
small~$\varepsilon$ contradicts $4k-5$ once $k\gtrsim 12$.

For $k\in\{2,3,4\}$, however, the asymptotic bound is not effective:
the leading term evaluates to approximately $1.26$, $2.84$, and $5.05$,
all well below the true values $3$, $7$, $11$.  Thus the asymptotic
bound leaves ample room for the small-$k$ pattern and, in particular,
for Conjecture~\ref{conj:C}.

\subsection{Comparison with the exponential upper bound}

\begin{proposition}\label{prop:bounds}
  For $k\in\{2,3,4\}$:
  \[
    \frac{k^2}{2\log_2 3}
    \;<\;
    \Rvm(k)
    \;\leq\;
    2^k - 1.
  \]
  Equality in the upper bound holds at $k\in\{2,3\}$ but fails at
  $k=4$.
\end{proposition}

\begin{proof}
  The upper bound is~\cite[Theorem~1.4]{AFFMP2026}.  The numerical
  comparisons are immediate from Corollary~\ref{cor:values}:
  $\Rvm(2)=3=2^2-1$, $\Rvm(3)=7=2^3-1$, and
  $\Rvm(4)=11<15=2^4-1$.
\end{proof}

\subsection{Explicit lower bounds from building blocks}\label{subsec:explicit-lower}

The asymptotic lower bound~\eqref{eq:lower-bound} does not yield
concrete numerical bounds for any particular~$k$.  The following
observation, combined with our computed values, produces the first
explicit lower bounds for~$\Rvm(k)$.

\begin{proposition}[Building-block lower bound]\label{prop:building}
  Let $H_1,\dots,H_r$ be graphs on pairwise disjoint vertex sets with
  $\beta(H_i)=j_i$, where
  $\beta(H)=\max_{H'\in[H]_{\LC}}\alpha(H')$.  If
  $j_1+\cdots+j_r\leq k-1$, then
  $\Rvm(k)>|H_1|+\cdots+|H_r|$.
\end{proposition}

\begin{proof}
  Set $G=H_1\cup\cdots\cup H_r$.  By
  Lemma~\ref{lem:disconnected},
  $\beta(G)=\sum_i\beta(H_i)\leq k-1<k$, so
  $E_k\not\vm G$ by Proposition~\ref{prop:alpha}.
\end{proof}

Three building blocks arise from our results:
$K_2$ with $\beta=1$;
$\overline{C_6}$ with $\beta=2$
(Proposition~\ref{prop:Rvm3}); and any $n=10$
counterexample~$G^*$ with $\beta=3$
(since $\Rvm(3)=7$ forces $\beta(G^*)\geq 3$, while Phase~3 for
$k=4$ gives $\beta(G^*)<4$).
Each block is optimal for its~$\beta$ value: $K_2$ is the largest
graph with $\beta=1$, $\overline{C_6}$ attains the maximum order~$6$
for $\beta=2$ (since $\Rvm(3)=7$), and the counterexamples attain the
maximum order~$10$ for $\beta=3$ (since $\Rvm(4)=11$).
Packing $\lfloor(k{-}1)/3\rfloor$ copies of~$G^*$ and filling
the remainder yields:

\begin{corollary}\label{cor:explicit}
  For every $k\geq 2$, writing $k-1=3q+r$ with $0\leq r\leq 2$,
  \[
    \Rvm(k) \;\geq\; 10q + v(r) + 1,
  \]
  where $v(0)=0$, $v(1)=2$, $v(2)=6$.
  In particular, $\Rvm(5)\geq 13$, $\Rvm(6)\geq 17$, and
  $\Rvm(7)\geq 21$.
\end{corollary}

\begin{proof}
  Apply Proposition~\ref{prop:building} with $q$~copies of~$G^*$
  ($\beta=3$ each).  If $r=1$, adjoin~$K_2$ ($\beta=1$); if $r=2$,
  adjoin $\overline{C_6}$ ($\beta=2$).  The total $\beta$ is
  $3q+r=k-1<k$ and the total order is $10q+v(r)$.
\end{proof}

Table~\ref{tab:explicit} compares these explicit bounds with the
leading term of~\eqref{eq:lower-bound}.  For $k\leq 9$,
Corollary~\ref{cor:explicit} is strictly stronger; the asymptotic bound
overtakes near $k=10$.

\begin{table}[ht]
  \centering
  \caption{Explicit lower bounds from Corollary~\ref{cor:explicit}
    vs.\ the leading term of the asymptotic lower
    bound~\eqref{eq:lower-bound}.  Entries marked ``tight'' match the
    known value of~$\Rvm(k)$.}
  \label{tab:explicit}
  \begin{tabular}{cccc}
    \toprule
    $k$ & Cor.~\ref{cor:explicit}
        & $\lfloor k^2/(2\log_2 3)\rfloor$
        & status \\
    \midrule
    $2$  & $3$   & $1$  & tight \\
    $3$  & $7$   & $2$  & tight \\
    $4$  & $11$  & $5$  & tight \\
    $5$  & $13$  & $7$  & new \\
    $6$  & $17$  & $11$ & new \\
    $7$  & $21$  & $15$ & new \\
    $8$  & $23$  & $20$ & new \\
    $9$  & $27$  & $25$ & new \\
    \bottomrule
  \end{tabular}
\end{table}

\subsection{Toward $\Rvm(5)$}\label{subsec:toward-k5}

We applied the three-phase classification with $k=5$ and an
orbit-search budget of $50{,}000$ to all non-isomorphic graphs on
$n\leq 10$ vertices.  For $n\in\{8,9\}$, a verification run with
budget~$10^7$ produced identical Phase~3 counts, confirming that the
orbit of every Phase~3 graph was fully explored within the original
budget; the counts below are therefore exact at those orders.  At
$n=10$, some orbits may exceed the budget, so the Phase~3 count for
$n=10$ is an upper bound on the number of graphs with $E_5\not\vm G$.

\begin{center}
  \begin{tabular}{crrrr}
    \toprule
    $n$ & total & Phase~1 & Phase~2 & Phase~3 \\
    \midrule
    $8$  & $12{,}346$       & $930$       & $1{,}498$     & $9{,}918$ \\
    $9$  & $274{,}668$      & $29{,}228$  & $61{,}564$    & $183{,}876$ \\
    $10$ & $12{,}005{,}168$ & $1{,}841{,}049$ & $4{,}699{,}331$ & $5{,}464{,}788$ \\
    \bottomrule
  \end{tabular}
\end{center}

For $k=4$ at $n=10$, only six graphs reach Phase~3; for $k=5$, more
than five million do.  This indicates that $\Rvm(5)$ is substantially
larger than~$11$, consistent with Conjecture~\ref{conj:C}.

\section{Conclusion and open problems}\label{sec:conclusion}

We determined $\Rvm(4)=11$ and classified the six extremal graphs on
$10$ vertices.  Up to isomorphism, they represent five
LC-equivalence classes, and each labeled LC orbit has size~$8{,}712$.
These extremal graphs serve as building blocks
(Corollary~\ref{cor:explicit}), yielding the first explicit lower
bounds on~$\Rvm(k)$ that surpass the leading term of the asymptotic
bound for all $k\leq 9$.

The next unresolved case is $\Rvm(5)$.  Our partial survey for $k=5$
(Section~\ref{subsec:toward-k5}) shows that millions of graphs on
$10$~vertices still avoid~$E_5$ as a vertex-minor, so $\Rvm(5)$ lies
well beyond the current computational frontier.  A direct extension to
$n=15$ would require classifying more than $3\times 10^{19}$
non-isomorphic graphs, so further progress will likely require
structural arguments rather than exhaustive enumeration.

A natural question is whether there is a structural characterization of
graphs that exclude $E_k$ as a vertex-minor.  The six counterexamples at
$n=10$ all have rank-width~$2$, clique number~$4$, and the same labeled
orbit size, which points to a more rigid structure than current theory
predicts.

Finally, since the vertex-minor relation is a well-quasi-order on
graphs of bounded rank-width~\cite{Oum2008}, one may ask how these
extremal examples sit inside that order and what they reveal about
vertex-minor ideals.

\section*{Acknowledgments}

Computations were performed using the \texttt{nauty} graph generation
tools~\cite{McKayPiperno2014}.  We thank the developers of
\texttt{nauty} and \texttt{networkx}.

\medskip
\noindent\textbf{Data availability.}
The source code, \texttt{graph6} counterexample files, and
cross-validation logs supporting this study are available from the
corresponding author upon reasonable request.

\providecommand{\bysame}{\leavevmode\hbox to3em{\hrulefill}\thinspace}
\providecommand{\MR}{\relax\ifhmode\unskip\space\fi MR }
\providecommand{\MRhref}[2]{%
  \href{http://www.ams.org/mathscinet-getitem?mr=#1}{#2}
}
\providecommand{\href}[2]{#2}

\end{document}